\documentclass[11pt]{amsart}

\usepackage[english]{babel}
\usepackage[utf8]{inputenc}
\usepackage[T1]{fontenc}
\usepackage{lmodern}
\usepackage{microtype}
\usepackage{amsmath,amssymb,mathtools}
\usepackage[dvipsnames]{xcolor}
\usepackage{hyperref}
\usepackage{tikz}
\usepackage{todonotes}
\usepackage{aliascnt}
\usepackage[nameinlink,capitalize]{cleveref}

\hypersetup{
  colorlinks=true,
  linkcolor=MidnightBlue,
  citecolor=ForestGreen,
  urlcolor=MidnightBlue,
  pdftitle={Formal Squares over the Unit Square Separate FS-Domains from RB-Domains},
  pdfauthor={Marco Abbadini},
  pdfsubject={Domain theory: FS-domains and RB-domains},
  pdfkeywords={domain theory, FS-domain, RB-domain, formal balls, formal squares, bounded variation}
}

\newtheorem{theorem}{Theorem}[section]

\newaliascnt{proposition}{theorem}
\newtheorem{proposition}[proposition]{Proposition}
\aliascntresetthe{proposition}

\newaliascnt{lemma}{theorem}
\newtheorem{lemma}[lemma]{Lemma}
\aliascntresetthe{lemma}

\newaliascnt{corollary}{theorem}
\newtheorem{corollary}[corollary]{Corollary}
\aliascntresetthe{corollary}

\newaliascnt{claim}{theorem}

\aliascntresetthe{claim}

\theoremstyle{definition}
\newaliascnt{definition}{theorem}
\newtheorem{definition}[definition]{Definition}
\aliascntresetthe{definition}

\newaliascnt{notation}{theorem}

\aliascntresetthe{notation}

\newaliascnt{example}{theorem}

\aliascntresetthe{example}

\theoremstyle{remark}
\newaliascnt{remark}{theorem}
\newtheorem{remark}[remark]{Remark}
\aliascntresetthe{remark}

\crefname{theorem}{Theorem}{Theorems}
\Crefname{theorem}{Theorem}{Theorems}
\crefname{proposition}{Proposition}{Propositions}
\Crefname{proposition}{Proposition}{Propositions}
\crefname{lemma}{Lemma}{Lemmas}
\Crefname{lemma}{Lemma}{Lemmas}
\crefname{corollary}{Corollary}{Corollaries}
\Crefname{corollary}{Corollary}{Corollaries}
\crefname{claim}{Claim}{Claims}
\Crefname{claim}{Claim}{Claims}
\crefname{definition}{Definition}{Definitions}
\Crefname{definition}{Definition}{Definitions}
\crefname{remark}{Remark}{Remarks}
\Crefname{remark}{Remark}{Remarks}
\crefname{appendix}{Appendix}{Appendices}
\Crefname{appendix}{Appendix}{Appendices}

\newcommand{\R}{\mathbb{R}}
\newcommand{\Z}{\mathbb{Z}}
\newcommand{\eps}{\varepsilon}
\newcommand{\id}{\mathrm{id}}
\newcommand{\dd}{\,d}
\newcommand{\one}{\mathbf{1}}
\newcommand{\Var}{\operatorname{Var}}
\newcommand{\Sq}{\mathbf{Sq}}
\newcommand{\Sqbot}{\mathbf{Sq}_{\bot}}
\newcommand{\abs}[1]{\left\lvert #1\right\rvert}

\newcommand{\Cen}{\mathrm{C}}
\newcommand{\Rad}{\mathrm{R}}
\newcommand{\exc}{\operatorname{exc}}
\newcommand{\domle}{\mathrel{\sqsubseteq}}

\usetikzlibrary{calc}

\title[Formal squares over the unit square]
{Formal squares over the unit square separate FS-domains from RB-domains}
\author{Marco Abbadini}

\begin{document}

\begin{abstract}
Every RB-domain is an FS-domain.  Whether the converse holds was a
long-standing open problem.
We prove that the domain of closed axis-parallel squares in the plane
whose centres lie in the unit square \([0,1]^2\), with the whole plane
adjoined and ordered by reverse inclusion, is an FS-domain but not an
RB-domain.

The proof is quantitative. 
A finite-grid argument
first shows that, on every finite slab $[0,1]^2\times[0,m]$, for every approximate identity and every $\varepsilon>0$, one member of the approximate
identity has radius excess (i.e., the output radius minus the input radius) uniformly at most $\varepsilon$.
By contrast, for every deflation and
every $m>0$, the radius excess is at least
$m/(4m+1)$
at some point of the slab $[0,1]^2\times[0,m]$.

This solution was obtained independently of the recent work of Chen, Kou, and Lyu and uses a different method.
\end{abstract}

\maketitle

\section{Introduction}

Every RB-domain is an FS-domain.  Whether the converse holds was a
long-standing open problem.
Positive answers were known for several
special classes; see, for example, \cite{ZouLiGuo2018}.  In his original
paper on FS-domains, Jung presented the domain of closed Euclidean discs,
with the whole plane adjoined as a least element and ordered by reverse
inclusion, as a concrete test case.  He attributed the example to Jimmie
Lawson and left open whether it is a retract of an SFP-object
(i.e., a retract of a countably based bifinite domain)
\cite[Section~4]{Jung1990}; the same example and open question were later
recorded in \cite[Section~4.2.2]{AbramskyJung1994}.  Lawson subsequently proved that, for every compact metric space \(X\),
the poset of closed formal balls with a bottom adjoined is an FS-domain
\cite[Proposition~1]{Lawson2008}.

Our main result is the following.

\begin{theorem}[Main theorem]\label{thm:main-intro}
For the metric space $[0,1]^2$ with the $\ell^\infty$ metric, the FS-domain of closed formal balls with a bottom adjoined is not an RB-domain.
\end{theorem}

The FS assertion follows from Lawson's theorem on formal-ball domains.
For the non-RB assertion, we first show that every approximate identity
contains a member which is uniformly close to the identity on each finite
slab.  We then prove, using a Minkowski-content slicing estimate, that every deflation has radius excess at least $m/(4m+1)$
somewhere on a slab of height $m$.  The two estimates are incompatible.

\begin{remark}[Comparison with recent work]
\label{rem:note-added}
The solution presented in this manuscript was obtained on 30 April 2026, with assistance from ChatGPT, as described in the Acknowledgments.

While I was subsequently revising the argument and preparing the manuscript
for publication, an independent solution to the problem of whether every FS-domain is an RB-domain appeared on
arXiv on 1 July 2026 \cite{ChenKouLyu2026}.  There, Chen, Kou, and Lyu proved that Lawson's planar
closed-disc domain is not an RB-domain, thereby
giving a negative answer to the long-standing question of
whether every FS-domain is an RB-domain.  I am very pleased to acknowledge
their work and to congratulate them on their solution.
I also mention that, on the next day, Chen's preprint \cite{ChenCone2026} exhibited that various other similar FS-domains are not RB-domains.

The proof developed here arose independently and follows a different method.  Since the core mathematical argument had been completed before the two preprints
\cite{ChenKouLyu2026,ChenCone2026} appeared, I hope that its
distinct perspective and techniques provide a useful complement to
those works.
\end{remark}

\section{Domain-theoretic setup and formal squares}\label{sec:domain}

\subsection{The notions used in the proof}

We recall only the required definitions.  Standard references are
\cite{AbramskyJung1994,GierzHofmannEtAl2003,Jung1990}.

\begin{definition}
A subset $A$ of a poset is \emph{directed} if every
finite subset of $A$ has an upper bound in $A$.  A \emph{dcpo} is a poset
in which every directed subset has a supremum.  It is \emph{pointed} if
it has a least element.
\end{definition}

\begin{definition}
Let $D$ be a dcpo.  For $x,y\in D$, one writes $x\ll y$ if, whenever
$A\subseteq D$ is directed and $y\leq\sup A$, there is $a\in A$ with
$x\leq a$.  The dcpo $D$ is \emph{continuous} if, for every $y\in D$,
the set $\{x \in D \mid x\ll y\}$ is directed and has supremum $y$.  A
continuous dcpo is also called a \emph{domain}.
\end{definition}

\begin{definition}
A map between dcpos is \emph{Scott-continuous} if it is monotone and
preserves directed suprema.  We write $[D\to E]$ for the set of
Scott-continuous maps from $D$ to $E$, ordered pointwise.  Directed
suprema in $[D\to E]$ are computed pointwise.
\end{definition}

\begin{definition}
Let $D$ be a pointed dcpo.
\begin{enumerate}
\item An \emph{approximate identity} on $D$ is a directed family
$\mathcal A\subseteq[D\to D]$ such that $\sup\mathcal A=\id_D$.

\item A Scott-continuous map $f \colon D\to D$ is \emph{finitely separated
from the identity} if there is a finite set $M\subseteq D$ such that,
for every $x\in D$, some $m\in M$ satisfies
\[
  f(x)\leq m\leq x.
\]

\item A \emph{deflation} is a Scott-continuous map
$\sigma \colon D\to D$ with finite image and $\sigma(x)\leq x$ for every
$x\in D$.
\end{enumerate}
\end{definition}

\begin{definition}
An \emph{FS-domain} is a pointed dcpo carrying an approximate identity
of maps finitely separated from the identity.
\end{definition}

An element $k$ of a domain is \emph{compact} if $k\ll k$, and a domain
is \emph{algebraic} if every element is the directed supremum of the
compact elements below it.  A \emph{bifinite domain} is an algebraic
FS-domain.  An \emph{RB-domain} (``retract of bifinite'') is a dcpo
isomorphic to a Scott-continuous retract of a bifinite domain.

Every deflation is finitely separated from the identity: take its finite
image as a separating set.  We use the following standard characterization
of RB-domains.

\begin{proposition}[Characterization of RB-domains]\label{prop:RB-char}
A pointed dcpo is an RB-domain if and only if it carries an approximate
identity consisting of deflations.
\end{proposition}

\begin{proof}
This is the standard characterization of retracts of bifinite domains;
see \cite[Exercise 4.3.11(9)(a)]{AbramskyJung1994} or
\cite[Lemma 2.4]{ZouLiGuo2018}.
\end{proof}

It follows at once that every RB-domain is an FS-domain.

\subsection{The formal-square domain}

On \(\R^2\), let
\[
  d_\infty((x,y),(x',y'))
  =\max\{\abs{x-x'},\abs{y-y'}\}.
\]
For $\mathbf{c}\in[0,1]^2$ and $r\geq0$, let
\[
  S(\mathbf{c},r)=\mathbf{c}+[-r,r]^2
\]
be the closed axis-parallel square of centre $\mathbf{c}$ and radius $r$.  Set
\[
  \Sq=[0,1]^2\times[0,\infty).
\]
Only the centre is constrained to \([0,1]^2\): the square \(S(\mathbf{c},r)\) is viewed
in the ambient plane and may extend outside \([0,1]^2\).
We use the symbol $\domle$ for the domain order, in order not to confuse
it with the usual order on radii:
\begin{equation}\label{eq:square-order}
  (\mathbf{c},r)\domle(\mathbf{c}',r')
  \quad\Longleftrightarrow\quad
  d_\infty(\mathbf{c},\mathbf{c}')\leq r-r'.
\end{equation}
Equivalently,
\[
  (\mathbf{c},r)\domle(\mathbf{c}',r')
  \quad\Longleftrightarrow\quad
  S(\mathbf{c},r)\supseteq S(\mathbf{c}',r').
\]
Thus the domain order is reverse inclusion.  Let
\[
  \Sqbot=\Sq\cup\{\bot\},
\]
where $\bot$ is a new least element, geometrically interpreted as the
whole plane.

\begin{remark}[The downward-skyscraper picture]
We represent the vertical coordinate \(t\in[0,\infty)\) on an axis whose
positive direction points downwards, and identify an element $(\mathbf{c},r)\in\Sq$ with the point
$(\mathbf{c},t)$ at depth $t=r$.
Thus the finite part of \(\Sqbot\) is represented by the downward infinite
prism \([0,1]^2\times[0,\infty)\).  Under this identification,
\[
  (\mathbf{c},t)\domle(\mathbf{c}',t')
  \quad\Longleftrightarrow\quad
  t\geq t'\ \text{ and }\ d_\infty(\mathbf{c},\mathbf{c}')\leq t-t'.
\]
One may
therefore picture the domain as a
skyscraper hanging downwards from the roof \([0,1]^2\times\{0\}\): the level at
depth \(r\) is the fixed-radius slice \([0,1]^2\times\{r\}\), and \(\bot\) lies
below the whole skyscraper.
\end{remark}

\begin{figure}[ht]
  \centering
  \begin{tikzpicture}[
    x={(5.0cm,0cm)},
    y={(1.15cm,0.82cm)},
    z={(0cm,-8.8cm)},
    line join=round,
    line cap=round
  ]
    %
    \def\vx{0.36}
    \def\vy{0.55}
    \def\rv{0.23}
    \def\depth{0.40}

    \coordinate (T00) at (0,0,0);
    \coordinate (T10) at (1,0,0);
    \coordinate (T11) at (1,1,0);
    \coordinate (T01) at (0,1,0);

    \coordinate (B00) at (0,0,\depth);
    \coordinate (B10) at (1,0,\depth);
    \coordinate (B11) at (1,1,\depth);
    \coordinate (B01) at (0,1,\depth);

    \coordinate (V) at (\vx,\vy,\rv);

    \coordinate (U00) at ({\vx-\rv},{\vy-\rv},0);
    \coordinate (U10) at ({\vx+\rv},{\vy-\rv},0);
    \coordinate (U11) at ({\vx+\rv},{\vy+\rv},0);
    \coordinate (U01) at ({\vx-\rv},{\vy+\rv},0);

    \fill[gray!10]
      (T00)--(T10)--(T11)--(T01)--cycle;

    \draw[gray!55,dashed]
      (B00)--(B10)--(B11)--(B01)--cycle;

    \draw[gray!55]
      (T00)--(B00)
      (T10)--(B10)
      (T11)--(B11)
      (T01)--(B01);

    %

    \fill[MidnightBlue!7]
      (V)--(U01)--(U11)--cycle;

    \fill[MidnightBlue!11]
      (V)--(U00)--(U01)--cycle;

    \fill[MidnightBlue!15]
      (V)--(U10)--(U11)--cycle;

    \fill[MidnightBlue!20]
      (V)--(U00)--(U10)--cycle;

    \fill[MidnightBlue!13]
      (U00)--(U10)--(U11)--(U01)--cycle;

    \draw[MidnightBlue!65,dashed]
      (V)--(U01)
      (V)--(U11);

    \draw[MidnightBlue,thick]
      (V)--(U00)
      (V)--(U10)
      (U00)--(U10)--(U11)--(U01)--cycle;

    \draw[gray!75,thick]
      (T00)--(T10)--(T11)--(T01)--cycle;

    \fill[MidnightBlue] (V) circle[radius=1.6pt];

    \node[MidnightBlue,below right=1pt] at (V)
      {$\mathbf{v}$};

    \node[MidnightBlue] at (0.43,0.43,0.085)
      {$\mathord{\uparrow}\mathbf{v}$};

    \node at (0.5,0.5,-0.115)
      {$[0,1]^2\times\{0\}$};

    \draw[->,thick]
      (-0.16,0,-0.015)--(-0.16,0,\depth)
      node[midway,left=5pt] {$t$};

    \node at (0.5,0.5,0.48) {$\vdots$};
    \node at (0.5,0.5,0.55) {$\bot$};
  \end{tikzpicture}

  \caption{The formal-square domain as a downward skyscraper.
  The shaded square cone is the principal upset
  \(\mathord{\uparrow}\mathbf{v}\) of the point \(\mathbf{v}\).}
\end{figure}
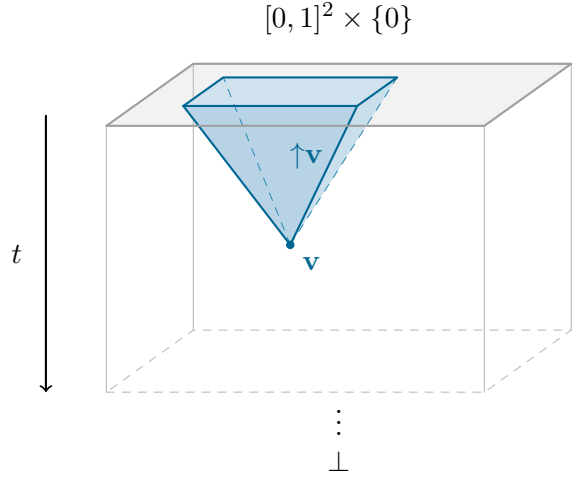

For $(\mathbf{c},r)\in\Sq$, write
\[
  \Cen(\mathbf{c},r)=\mathbf{c},
  \qquad
  \Rad(\mathbf{c},r)=r;
\]
$\Cen$ stands for ``centre'', while $\Rad$ stands for ``radius''.
Set $\Rad(\bot)=\infty$; the centre of $\bot$ is left undefined.  We call
the elements of $\Sq$ \emph{finite squares}, meaning that they are not $\bot$.

\begin{remark}[The radius is order-reversing]\label{rem:radius-antitone}
For all $u,v\in\Sqbot$,
\[
  u\domle v\quad\Longrightarrow\quad \Rad(u)\geq\Rad(v),
\]
where $\Rad(\bot)=\infty$.  This simple observation is used repeatedly:
moving upward in the domain order makes the radius smaller.
\end{remark}

\begin{proposition}[The formal-square dcpo]\label{prop:way-below}
The poset $\Sqbot$ is a pointed continuous dcpo.  Its new least element
satisfies $\bot\ll u$ for every $u\in\Sqbot$, and, for finite squares,
\[
  (\mathbf{c},r)\ll(\mathbf{c}',r')
  \quad\Longleftrightarrow\quad
  d_\infty(\mathbf{c},\mathbf{c}')<r-r'.
\]
In particular, the finite squares way below a given finite square are
cofinal among all its way-below approximants.
\end{proposition}

\begin{proof}
The formal-ball poset of a complete metric space is a continuous dcpo,
and its way-below relation between finite formal balls is the strict
formal-ball inequality displayed above; see \cite[Section 2]{Lawson2008}
or \cite[Chapter V-6]{GierzHofmannEtAl2003}.  Adjoining a new least
element preserves directed suprema and the way-below relation between
the old elements.  The new element is way below every element because
directed sets are nonempty.  These observations also prove the stated
cofinality.
\end{proof}

\begin{proposition}[Lawson]\label{prop:FS}
The pointed formal-square domain $\Sqbot$ is an FS-domain.
\end{proposition}

\begin{proof}
Lawson proved that, if $(X,d)$ is complete and every closed bounded set is
compact, then the pointed formal-ball domain $B_X^\bot$ is an FS-domain
\cite[Proposition 1 and Remark 2(ii)--(iii)]{Lawson2008}.  Apply this to
$X=([0,1]^2,d_\infty)$, which is compact.  Its formal-ball pairs are precisely
$[0,1]^2\times[0,\infty)$, with the order in \eqref{eq:square-order}; hence
$B_{[0,1]^2}^\bot\cong\Sqbot$.
\end{proof}

It remains to prove that $\Sqbot$ admits no approximate identity of
deflations.

\section{The main idea}\label{sec:main-idea}

We present informally the main idea of the proof.
Suppose, towards a contradiction, that $\Sqbot$ carries an approximate
identity $\mathcal A$ consisting of deflations.  The first conclusion we
shall draw is that, for every finite slab
\[
  [0,1]^2\times[0,m]
\]
and every $\varepsilon>0$, some $\sigma\in\mathcal A$ has radius excess at
most $\varepsilon$ everywhere on that slab.  The definition of an
approximate identity gives only pointwise approximation, but a finite grid
placed slightly below the slab upgrades this to uniform approximation.

A small radius excess also forces the output centre to remain close to the
input centre.  Indeed, if
\[
  \sigma(\mathbf{c},r)=(\gamma_r(\mathbf{c}),\rho_r(\mathbf{c})),
\]
then the deflation inequality gives
\[
  d_\infty(\gamma_r(\mathbf{c}),\mathbf{c})
  \leq \rho_r(\mathbf{c})-r.
\]
Thus the member of the approximate identity just chosen moves all centres
by at most $\varepsilon$ on the slab, and in particular it only slightly
shrinks its lateral boundary.

We now ask what the fact that $\sigma$ is a deflation forces in the
opposite direction.  For each \(r\in[0,\infty)\), consider the restriction of
\(\sigma\) to the fixed-depth slice $[0,1]^2\times\{r\}$:
\begin{align*}
  \sigma_{\upharpoonright r}\colon [0,1]^2
    &\longrightarrow \Sqbot,\\
  \mathbf{c}
    &\longmapsto \sigma(\mathbf{c},r).
\end{align*}
Define the output-radius function on this slice by
\begin{align*}
  \rho_r\colon [0,1]^2
    &\longrightarrow [0,\infty],\\
  \mathbf{c}
    &\longmapsto \Rad\bigl(\sigma_{\upharpoonright r}(\mathbf{c})\bigr),
\end{align*}
and consider its average value
\[
  A(r)
  =\int_{[0,1]^2}\rho_r(\mathbf{c})\dd\mathbf{c}.
\]
Since the unit square has area \(1\), this integral is indeed the average
value of the output radius on the slice.

The guiding idea is that, as long as the lateral boundary is not shrunk,
the function \(r\mapsto A(r)\) is forced to grow with right-hand slope at least \(2\).
Since
\(A(0)\geq0\), we would obtain
\[
  A(r)\geq2r.
\]
The identity map has radius exactly \(r\) everywhere on the slice at depth
\(r\), and hence average radius \(r\).  Thus \(\sigma\) would be, on
average, at least twice as deep as the identity, and therefore could not
be close to the identity.

The slope \(2\) is obtained when \(\sigma\) does not ``shrink'' the
lateral boundary of the prism.
If $\sigma$ shrinks the boundaries by a small number $\delta$, then one obtains the growth of $A$ is bounded from below by $2(1 - 2\delta)$, i.e., by a value close to $2$.

From this one concludes that an approximate identity made of deflation cannot exist; indeed, if it existed, fix $m > 0$; then there would be an element $\sigma$ of the approximate identity with small radius excess on $[0,1]^2 \times [0,m]$. This would shrink just a little the boundaries, and so the slope of $A(r)$ would be close to $2$, which forces a big radius excess.

That is the basic idea.  The main difficulty in implementing it
is showing that the average radius grows with slope close to \(2\).

Let us conclude this informal discussion with a model illustrating
the appearance of the slope \(2\).  Suppose, purely for illustration, that on the
top slice \([0,1]^2\times\{0\}\) a deflation acts as follows:
\[
\sigma((a,b),0)=
\begin{cases}
  \left(\left(\frac12,\frac12\right),\frac12\right),
    & a,b\in(0,1),\\[1ex]
  \left(\left(\frac12,b\right),1\right),
    & a\in(0,1),\ b\in\{0,1\},\\[1ex]
  \left(\left(a,\frac12\right),1\right),
    & a\in\{0,1\},\ b\in(0,1),\\[1ex]
  \left(\left(a,b\right),\frac32\right),
    & a,b\in\{0,1\}.
\end{cases}
\]
Thus the open square, the four open edges, and the four vertices form nine
strata, carrying the radii displayed in
\cref{fig:main-idea-top-slice}.

\begin{figure}[ht]
  \centering
  \begin{tikzpicture}[
    x=5.2cm,
    y=5.2cm,
    line cap=round,
    line join=round
  ]
    \fill[gray!10] (0,0) rectangle (1,1);

    \draw[MidnightBlue,line width=2.4pt]
      (0,0)--(1,0)
      (1,0)--(1,1)
      (1,1)--(0,1)
      (0,1)--(0,0);

    \draw[black,thick] (0,0) rectangle (1,1);

    \node[fill=gray!10,inner sep=1.5pt] at (0.5,0.5)
      {$\frac12$};

    \node[below=2pt] at (0.5,0) {$1$};
    \node[above=2pt] at (0.5,1) {$1$};
    \node[left=2pt]  at (0,0.5) {$1$};
    \node[right=2pt] at (1,0.5) {$1$};

    \foreach \x/\y in {0/0,1/0,1/1,0/1}
      \fill[BrickRed] (\x,\y) circle[radius=2.3pt];

    \node[below left=2pt]  at (0,0) {$\frac32$};
    \node[below right=2pt] at (1,0) {$\frac32$};
    \node[above right=2pt] at (1,1) {$\frac32$};
    \node[above left=2pt]  at (0,1) {$\frac32$};
  \end{tikzpicture}

  \caption{The prescribed output radius on the top slice.  The nine
  strata are the open square, the four open edges, and the four
  vertices.}
  \label{fig:main-idea-top-slice}
\end{figure}
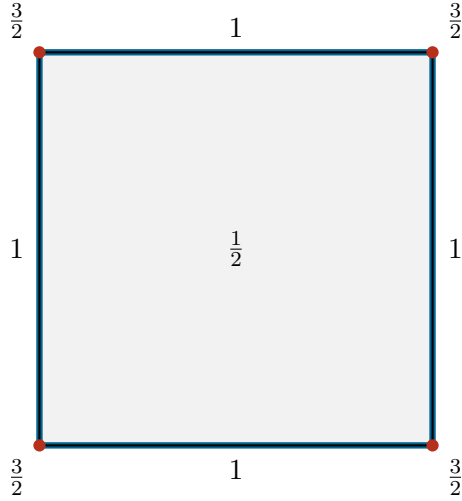

This choice is compatible with some information we know about a slice of
a deflation; it is the smallest possible choice of output radii compatible
with choosing as output centre of each of the nine regions the centre of
that region.

Let
\[
  \rho_{0}(a,b)=\Rad\bigl(\sigma(a,b,0)\bigr).
\]
The four edges and the four vertices have two-dimensional Lebesgue
measure zero, and therefore the average radius on the top slice is
\[
  A(0)
  =\int_{[0,1]^2}\rho_{0}(a,b)\dd a\dd b
  =\frac12.
\]

Now consider the slice at depth \(\varepsilon\), where
\(0<\varepsilon<1/2\).  Suppose that \((a,b)\) is within
\(d_\infty\)-distance \(\varepsilon\) of a point \((c,d)\) for which
\(\sigma(c,d,0)\) has radius \(\alpha\).  Then
\[
  (a,b,\varepsilon)\domle(c,d,0).
\]
By monotonicity of \(\sigma\),
\[
  \sigma(a,b,\varepsilon)\domle\sigma(c,d,0),
\]
and, since the radius is order-reversing,
\[
  \Rad\bigl(\sigma(a,b,\varepsilon)\bigr)\geq\alpha.
\]
Thus every point on the deeper slice inherits at least the largest radius
seen within \(d_\infty\)-distance \(\varepsilon\) on the top slice.  This gives the lower bound represented in
\cref{fig:main-idea-epsilon-slice}.

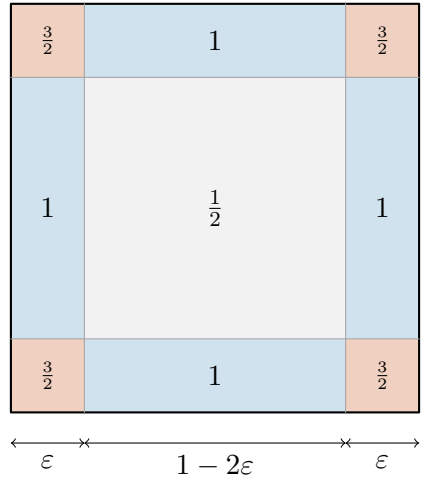
\begin{figure}[ht]
  \centering
  \begin{tikzpicture}[
    x=5.4cm,
    y=5.4cm,
    line cap=round,
    line join=round
  ]
    \def\e{0.18}

    \fill[gray!10]
      (\e,\e) rectangle ({1-\e},{1-\e});

    \fill[MidnightBlue!13]
      (\e,0) rectangle ({1-\e},\e);
    \fill[MidnightBlue!13]
      (\e,{1-\e}) rectangle ({1-\e},1);
    \fill[MidnightBlue!13]
      (0,\e) rectangle (\e,{1-\e});
    \fill[MidnightBlue!13]
      ({1-\e},\e) rectangle (1,{1-\e});

    \fill[BrickRed!18] (0,0) rectangle (\e,\e);
    \fill[BrickRed!18] ({1-\e},0) rectangle (1,\e);
    \fill[BrickRed!18] ({1-\e},{1-\e}) rectangle (1,1);
    \fill[BrickRed!18] (0,{1-\e}) rectangle (\e,1);

    \draw[black,thick] (0,0) rectangle (1,1);
    \draw[gray!70]
      (\e,0)--(\e,1)
      ({1-\e},0)--({1-\e},1)
      (0,\e)--(1,\e)
      (0,{1-\e})--(1,{1-\e});

    \node at (0.5,0.5) {$\frac12$};

    \node at (0.5,{\e/2}) {$1$};
    \node at (0.5,{1-\e/2}) {$1$};
    \node at ({\e/2},0.5) {$1$};
    \node at ({1-\e/2},0.5) {$1$};

    \node[font=\scriptsize] at ({\e/2},{\e/2}) {$\frac32$};
    \node[font=\scriptsize] at ({1-\e/2},{\e/2}) {$\frac32$};
    \node[font=\scriptsize] at ({1-\e/2},{1-\e/2}) {$\frac32$};
    \node[font=\scriptsize] at ({\e/2},{1-\e/2}) {$\frac32$};

    \draw[<->] (0,-0.075)--(\e,-0.075)
      node[midway,below=1pt] {$\varepsilon$};
    \draw[<->] (\e,-0.075)--({1-\e},-0.075)
      node[midway,below=1pt] {$1-2\varepsilon$};
    \draw[<->] ({1-\e},-0.075)--(1,-0.075)
      node[midway,below=1pt] {$\varepsilon$};
  \end{tikzpicture}

  \caption{The lower bound on the output radius at depth
  \(\varepsilon\).  Each displayed number is a lower bound on the
  corresponding region.}
  \label{fig:main-idea-epsilon-slice}
\end{figure}

The average of the displayed lower bound is
\[
  \frac12(1-2\varepsilon)^2
  +4\varepsilon(1-2\varepsilon)
  +4\cdot\frac32\varepsilon^2 =\frac12+2\varepsilon.
\]
Consequently,
\[
  A(\varepsilon)\geq\frac12+2\varepsilon
  =A(0)+2\varepsilon.
\]
Subtracting \(A(0)=1/2\), dividing by \(\varepsilon\), and letting
\(\varepsilon\downarrow0\), we obtain the lower bound \(2\) for the
right-hand slope.  This is exactly the number we wanted to explain.

Notice that, in computing the increase of the displayed lower bound
over \(\rho_0\), only the regions on which that lower bound differs
from \(\rho_0\) contribute.  Roughly speaking, one may picture the square as
being divided into finitely many strata: some of dimension \(2\), some
of dimension \(1\), and some of dimension \(0\).  Distinct strata of a
given dimension meet, if at all, through lower-dimensional strata: the
one-dimensional pieces form interfaces between two-dimensional pieces,
and the zero-dimensional pieces occur where one-dimensional pieces
meet.  When a stratum touches a higher-dimensional one, its larger radius
spreads into an \(\varepsilon\)-neighbourhood and contributes to the
change of average.

The contribution of the zero-dimensional strata is of order
\(\varepsilon^2\), and is therefore negligible at first order.  The
one-dimensional strata have \(\varepsilon\)-neighbourhoods of area of
order \(\varepsilon\), and provide the first-order contribution.
For example, in our model, each of the four sides of the boundary
contributes $(1 - \frac{1}{2})\eps$, i.e., $\frac{1}{2}\eps$.  In total,
these four quantities sum to $2\eps$.
The crucial hope---which is what works---is that this contribution can be
split into a horizontal part and a vertical part, reducing the planar
problem to two one-dimensional problems.
In fact, in our example, the two horizontal sides of the square contribute
to the vertical part (with a total vertical contribution of $\eps$), and
the two vertical sides contribute to the horizontal part (with a total
horizontal contribution of $\eps$).  In each direction, the
first-order contribution is \(\varepsilon\) times a total variation.  If
the boundary is not shrunk, the relevant horizontal variation is at
least \(1\), and the relevant vertical variation is also at least \(1\).
The local order constraint transfers these variations of the
output-centre map to the radius function.  The two contributions
therefore add up to
\[
  \varepsilon+\varepsilon=2\varepsilon.
\]
This is the geometric origin of the slope \(2\).
The argument can be repeated analogously by
tessellating $[0,1]^2$ with $n \times n$ squares.
The analytic argument
below makes this heuristic rigorous, without assuming that the level
sets of an arbitrary deflation admit such a regular stratification.

\section{Uniform approximation on finite slabs}
\label{sec:uniform-slab}

We first make precise the consequence of an approximate identity described
in the preceding section.  Pointwise approximation on a suitable finite
grid will give uniform approximation on an entire finite slab.

For a function $\sigma\colon\Sqbot\to\Sqbot$ below the identity, define
the extended radius excess by
\[
  \exc_\sigma(\mathbf{c},r)=\Rad(\sigma(\mathbf{c},r))-r\in[0,\infty],
\]
with the convention $\infty-r=\infty$.  Since
$\sigma(\mathbf{c},r)\domle(\mathbf{c},r)$, \cref{rem:radius-antitone} gives
$\exc_\sigma(\mathbf{c},r)\geq0$.  Whenever the output is finite, the full order
condition also gives
\begin{equation}\label{eq:basic-center-excess}
  d_\infty\bigl(\Cen(\sigma(\mathbf{c},r)),\mathbf{c}\bigr)
  \leq \exc_\sigma(\mathbf{c},r).
\end{equation}
Thus a small (finite) radius excess automatically keeps the output centre
close to the input centre.

We begin with the direct consequence of the definition of an approximate
identity.

\begin{lemma}[Approximation on a finite test set]\label{lem:finite-test-approx}
Let $\mathcal A$ be an approximate identity on $\Sqbot$.  For every
finite set $F\subseteq\Sq$ and every $\eps>0$, there is
$\sigma\in\mathcal A$ such that, for every $(\mathbf{c},r)\in F$,
\[
  r \leq \Rad(\sigma(\mathbf{c},r))\leq r+\eps.
\]
\end{lemma}

\begin{proof}
The case $F=\varnothing$ is immediate.  Fix $(\mathbf{c},r)\in F$.  By
\cref{prop:way-below}, $(\mathbf{c},r+\eps)\ll(\mathbf{c},r)$.  Since the
supremum of the approximate identity is computed pointwise,
\[
  (\mathbf{c},r)=\sup_{\sigma\in\mathcal A}\sigma(\mathbf{c},r),
\]
there is $\sigma_{(\mathbf{c},r)}\in\mathcal A$ with
\[
  (\mathbf{c},r+\eps)\domle\sigma_{(\mathbf{c},r)}(\mathbf{c},r).
\]
By \cref{rem:radius-antitone},
\[
  \Rad(\sigma_{(\mathbf{c},r)}(\mathbf{c},r))\leq r+\eps.
\]
The family $\mathcal A$ is directed and $F$ is finite, so there is a
single $\sigma\in\mathcal A$ above every $\sigma_{(\mathbf{c},r)}$.
Hence
\[
  \Rad(\sigma(\mathbf{c},r))
  \leq
  \Rad(\sigma_{(\mathbf{c},r)}(\mathbf{c},r))
  \leq r+\eps.
\]
Moreover, every element of $\mathcal A$ lies below its supremum, so
\[
  \sigma\leq\sup\mathcal A=\id_{\Sqbot}
\]
pointwise.  Thus $\sigma(\mathbf{c},r)\domle(\mathbf{c},r)$, and
\cref{rem:radius-antitone} gives
\[
  r\leq\Rad(\sigma(\mathbf{c},r)). \qedhere
\]
\end{proof}

The next theorem is the first of the two conclusions that will eventually
be incompatible.  Notice that it applies to an arbitrary approximate
identity; its members need not have finite image.

\begin{theorem}[Uniform approximation on finite slabs]
\label{thm:uniform-slab-approximation}
Let $\mathcal A$ be an approximate identity on $\Sqbot$.  For every
$m\in[0,\infty)$ and every $\eps>0$, there is $\sigma\in\mathcal A$ such
that, for every $(\mathbf{c},r)\in[0,1]^2\times[0,m]$,
\begin{equation}\label{eq:uniform-radius-on-slab}
  0
  \leq
  \Rad(\sigma(\mathbf{c},r))-r
  \leq
  \eps.
\end{equation}
In particular, all these outputs are finite and
\begin{equation}\label{eq:uniform-centre-on-slab}
  d_\infty\bigl(
    \Cen(\sigma(\mathbf{c},r)),
    \mathbf{c}
  \bigr)
  \leq
  \Rad(\sigma(\mathbf{c},r))-r
  \leq
  \eps
\end{equation}
throughout the slab.
\end{theorem}

\begin{proof}
Set $h=\eps/3$ and define
\[
  G_h
  =
  \bigl((2h\Z)\cap[0,1]\bigr)\cup\{1\},
  \qquad
  \Lambda_{m,h}
  =
  G_h^2\times\bigl(h\Z\cap[0,m+2h]\bigr).
\]
This is a finite subset of $\Sq$.  By
\cref{lem:finite-test-approx}, there is $\sigma\in\mathcal A$ such that
\begin{equation}\label{eq:grid-radius-control}
  \Rad(\sigma(\mathbf{c}_0,r_0))\leq r_0+h
  \qquad
  ((\mathbf{c}_0,r_0)\in\Lambda_{m,h}).
\end{equation}

Fix $(\mathbf{c},r)\in[0,1]^2\times[0,m]$.  Choose
$\mathbf{c}_0\in G_h^2$ with
\[
  d_\infty(\mathbf{c},\mathbf{c}_0)\leq h
\]
and choose
\[
  r_0\in h\Z\cap[r+h,r+2h].
\]
Such choices exist because consecutive points of $G_h$ are at distance at
most $2h$, while the radius lattice has spacing $h$.  Moreover,
$r_0\leq m+2h$, so $(\mathbf{c}_0,r_0)\in\Lambda_{m,h}$.  We have
\[
  d_\infty(\mathbf{c},\mathbf{c}_0)
  \leq h
  \leq r_0-r,
\]
and therefore
\[
  (\mathbf{c}_0,r_0)\domle(\mathbf{c},r).
\]
By monotonicity of $\sigma$, \cref{rem:radius-antitone}, and
\eqref{eq:grid-radius-control},
\[
  \Rad(\sigma(\mathbf{c},r))
  \leq
  \Rad(\sigma(\mathbf{c}_0,r_0))
  \leq
  r_0+h
  \leq
  r+3h
  =r+\eps.
\]
Every member of $\mathcal A$ lies below $\id_{\Sqbot}$, as observed in the
proof of \cref{lem:finite-test-approx}, and hence the radius excess is
nonnegative.  This proves \eqref{eq:uniform-radius-on-slab}.  In
particular, the outputs are finite, and \eqref{eq:basic-center-excess}
then gives \eqref{eq:uniform-centre-on-slab}.
\end{proof}

\section{Boundary preservation forces radial growth}
\label{sec:deflations}

We now turn to the opposite conclusion.  This section studies a single
deflation which is finite on a slab and whose output centres stay close
to the input centres on the lateral boundary.  We first describe the
local geometry forced by Scott continuity and finite image.

\begin{definition}[Slice at a fixed depth]\label{def:fixed-depth-slice}
Let $\sigma\colon\Sqbot\to\Sqbot$ be a function and let
$r\in[0,\infty)$.  The \emph{slice of $\sigma$ at depth $r$} is the
function
\begin{align*}
  \sigma_{\upharpoonright r}\colon [0,1]^2 &\longrightarrow \Sqbot,\\
  \mathbf{c} &\longmapsto \sigma(\mathbf{c},r).
\end{align*}
\end{definition}

\begin{lemma}[Local centre motion forces radial variation]
\label{lem:local-order}
Let $\sigma\colon\Sqbot\to\Sqbot$ be Scott-continuous with finite image.
Fix $r \in [0,\infty)$.  Then every element of $[0,1]^2$ is a local
minimum for the slice $\sigma_{\upharpoonright r}$ with respect
to the relative topology.  In fact, the following stronger statement
holds: for every $\mathbf{c}_0 \in[0,1]^2$ such that
$\sigma_{\upharpoonright r}(\mathbf{c}_0)$ is finite, there is
a relative open neighbourhood $U$ of $\mathbf{c}_0$ such that, for all $\mathbf{c} \in U$,
$\sigma_{\upharpoonright r}(\mathbf{c})$ is finite and
\begin{equation}\label{eq:local-centre-radius}
  d_\infty\bigl(
    \Cen(\sigma_{\upharpoonright r}(\mathbf{c}_0)),
    \Cen(\sigma_{\upharpoonright r}(\mathbf{c}))
  \bigr)
  \leq
  \Rad(\sigma_{\upharpoonright r}(\mathbf{c}_0))
  -\Rad(\sigma_{\upharpoonright r}(\mathbf{c})).
\end{equation}
\end{lemma}

Before proving the lemma, I invite the reader to read
\eqref{eq:local-centre-radius} as a constraint that forces
$\Rad(\sigma_{\upharpoonright r}(\mathbf{c}_0))$ to be large if there
are points $\mathbf{c}$ arbitrarily close to $\mathbf{c}_0$ for which
the centre $\Cen(\sigma_{\upharpoonright r}(\mathbf{c}))$ is far from
$\Cen(\sigma_{\upharpoonright r}(\mathbf{c}_0))$.

\begin{proof}
Fix $\mathbf{c}_0\in[0,1]^2$. For every $n\geq1$, put
\[
  u_n\coloneqq(\mathbf{c}_0,r+1/n).
\]
The sequence $(u_n)_{n\geq1}$ is an increasing chain in $\Sqbot$ with
supremum $(\mathbf{c}_0,r)$. Hence, by Scott continuity,
\[
  \sigma_{\upharpoonright r}(\mathbf{c}_0)
  =
  \sigma(\mathbf{c}_0,r)
  =
  \sup_{n\geq1}\sigma(u_n).
\]
The increasing chain $(\sigma(u_n))_{n\geq1}$ is contained in the finite
image of $\sigma$, and therefore eventually stabilises. Its eventual value
must be its supremum, so we may choose $n\geq1$ such that
\[
  \sigma(u_n)=\sigma_{\upharpoonright r}(\mathbf{c}_0).
\]
Set
\[
  U\coloneqq
  \{\mathbf{c}\in[0,1]^2\mid d_\infty(\mathbf{c},\mathbf{c}_0)<1/n\}.
\]
For every $\mathbf{c}\in U$, we have
\[
  d_\infty(\mathbf{c}_0,\mathbf{c})
  <
  \frac1n
  =
  \left(r+\frac1n\right)-r,
\]
so \cref{prop:way-below} gives
\[
  u_n\ll(\mathbf{c},r).
\]
In particular, $u_n\domle(\mathbf{c},r)$. By monotonicity of $\sigma$,
\[
  \sigma_{\upharpoonright r}(\mathbf{c}_0)
  =
  \sigma(u_n)
  \domle
  \sigma(\mathbf{c},r)
  =
  \sigma_{\upharpoonright r}(\mathbf{c})
  \qquad(\mathbf{c}\in U).
\]
Thus $\mathbf{c}_0$ is a local minimum of $\sigma_{\upharpoonright r}$.

Suppose now that $\sigma_{\upharpoonright r}(\mathbf{c}_0)$ is finite.
Since the only nonfinite
element of $\Sqbot$ is $\bot$, and $\bot$ does not lie above any finite
square, the preceding inequality implies that
$\sigma_{\upharpoonright r}(\mathbf{c})$ is finite for
every $\mathbf{c}\in U$. The definition of the formal-square order then gives
\[
  d_\infty\bigl(
    \Cen(\sigma_{\upharpoonright r}(\mathbf{c}_0)),
    \Cen(\sigma_{\upharpoonright r}(\mathbf{c}))
  \bigr)
  \leq
  \Rad(\sigma_{\upharpoonright r}(\mathbf{c}_0))
  -\Rad(\sigma_{\upharpoonright r}(\mathbf{c})),
\]
which is \eqref{eq:local-centre-radius}.
\end{proof}

By \cref{lem:local-order}, any motion of the output centres on a slice
must be paid for by oscillation of the output radii.  When the output
centres stay close to the identity on the boundary, that motion is
substantial.  Order preservation then forces the average output radius
on infinitesimally deeper slices to grow rapidly.  We now make this
precise.

\subsection{Infinitesimal radial growth}\label{sec:analytic}

For a function $f \colon [0,1]^2\to [0, \infty]$ with finite image and $c \geq 0$, define its
local upper envelope by
\begin{align*}
	M_c f \colon [0,1]^2 & \longrightarrow [0, \infty]\\
	\mathbf{z} & \longmapsto \max\{f(\widetilde{\mathbf{z}}) \mid \widetilde{\mathbf{z}}\in[0,1]^2,
              d_\infty(\mathbf{z},\widetilde{\mathbf{z}})\leq c\}.
\end{align*}
The maximum exists because the image of $f$ is finite.
Notice that $M_c f$ is pointwise larger than $f$; moreover, roughly speaking, the more variation $f$ has, the bigger will be the difference between $M_c f$ and $f$.
The reason why this is pertinent is the following.

\begin{lemma}[Big variation implies big growth on the attained radii]
\label{lem:radial-upper-envelope}
Let $\sigma\colon\Sqbot\to\Sqbot$ be an order-preserving function
with finite image. For $s\in[0,\infty)$, define
\begin{align*}
  \rho_s\colon [0,1]^2 & \longrightarrow[0,\infty],\\
  \mathbf{c} &\longmapsto \Rad\bigl(\sigma(\mathbf{c},s)\bigr).
\end{align*}
If $0\leq r\leq r' < \infty$, then
\[
  \rho_{r'}\geq M_{r'-r}\rho_r
\]
pointwise on $[0,1]^2$.
\end{lemma}

\begin{proof}
Fix $\mathbf{c}\in[0,1]^2$. Let $\widetilde{\mathbf{c}}\in[0,1]^2$ satisfy
\[
  d_\infty(\mathbf{c},\widetilde{\mathbf{c}})\leq r'-r.
\]
By the definition of the formal-square order,
\[
  (\mathbf{c},r')\domle(\widetilde{\mathbf{c}},r).
\]
Since $\sigma$ is order-preserving,
\[
  \sigma(\mathbf{c},r')\domle\sigma(\widetilde{\mathbf{c}},r).
\]
The radius is order-reversing with respect to $\domle$, by
\cref{rem:radius-antitone}; hence
\[
  \rho_{r'}(\mathbf{c})
  =
  \Rad\bigl(\sigma(\mathbf{c},r')\bigr)
  \geq
  \Rad\bigl(\sigma(\widetilde{\mathbf{c}},r)\bigr)
  =
  \rho_r(\widetilde{\mathbf{c}}).
\]
Since this holds for every $\widetilde{\mathbf{c}}\in[0,1]^2$ such that
$d_\infty(\mathbf{c},\widetilde{\mathbf{c}})\leq r'-r$, taking the maximum gives
\[
  \rho_{r'}(\mathbf{c})
  \geq
  \max\bigl\{
    \rho_r(\widetilde{\mathbf{c}})
    \mid
    \widetilde{\mathbf{c}}\in[0,1]^2,\
    d_\infty(\mathbf{c},\widetilde{\mathbf{c}})\leq r'-r
  \bigr\}
  =
  M_{r'-r}\rho_r(\mathbf{c}). \qedhere
\]
\end{proof}

The idea is the following: on every slice on which the output-centre map
has large boundary variation, \cref{lem:local-order} forces the associated
radius function $\rho_r$ to have large variation.  Its local upper envelope
$M_c\rho_r$ must therefore grow rapidly in $c$; the precise average growth
estimate is the content of \cref{thm:analytic-engine}.  Then, by
\cref{lem:radial-upper-envelope}, this means that the average radius
attained by $\sigma$ on a slice at depth $c$ must grow fast in $c$ (about
twice as fast as $c$).

The following theorem is the analytic statement needed in the proof.
It is useful to read its assumptions as follows.  The map $\gamma$ records
centres.  The scalar function $\rho$ records radii.  The local inequality
says that any local motion of $\gamma$ must be paid for by a decrease of
$\rho$.  For the final estimate, it is enough that $\gamma$ stay close to
the identity on the boundary of the square.

\begin{theorem}[Near-identity centre maps force radial oscillation]
\label{thm:analytic-engine}
Let $\sigma = (\gamma, \rho) \colon [0,1]^2 \to \Sq$ be a function with finite image such that, for every $\mathbf{z}\in[0,1]^2$, there is a relative neighbourhood
$U$ of $\mathbf{z}$ such that, for every $\widetilde{\mathbf{z}}\in U$,
\begin{equation}\label{eq:local-domination}
  d_\infty(\gamma(\mathbf{z}),\gamma(\widetilde{\mathbf{z}}))
  \leq \rho(\mathbf{z})-\rho(\widetilde{\mathbf{z}}).
\end{equation}
Set
\[
  \Delta_{v}
  \coloneqq 
  \int_0^1
  \abs{\gamma_1(1,y)-\gamma_1(0,y)}
  \dd y,
  \qquad
  \Delta_{h}
  \coloneqq
  \int_0^1
  \abs{\gamma_2(x,1)-\gamma_2(x,0)}
  \dd x.
\]
Then
\begin{align}
  \liminf_{\eps\downarrow0}
  \int_{[0,1]^2}\frac{M_\eps \rho(\mathbf{z})-\rho(\mathbf{z})}{\eps}\dd\mathbf{z}
  &\geq
  \int_0^1\Var(\rho(\cdot,y))\dd y
  +
  \int_0^1\Var(\rho(x,\cdot))\dd x
  \notag\\
  &\geq
  \Delta_{v}+\Delta_{h}.
  \label{eq:analytic-engine-conclusion}
\end{align}
In particular, for any $\delta \in [0, \infty)$, if
\[
  d_\infty(\gamma(\mathbf{z}),\mathbf{z})\leq\delta
  \qquad(\mathbf{z}\in\partial([0,1]^2)),
\]
then
\[
  \liminf_{\eps\downarrow0}
  \int_{[0,1]^2}\frac{M_\eps \rho(\mathbf{z})-\rho(\mathbf{z})}{\eps}\dd\mathbf{z}
  \geq 2(1-2\delta).
\]
\end{theorem}

The proof is given below, after two auxiliary results.

The main task is to explain
why the first-order growth of $M_\eps \rho$ sees the sum of these horizontal
and vertical variations.  This is precisely the content of the
Minkowski-slicing estimate below.

For a one-variable function $g \colon [0,1]\to\R$, its total variation is
\[
  \Var(g)
  =\sup_{0=t_0<\cdots<t_n=1}
    \sum_{j=1}^n\abs{g(t_j)-g(t_{j-1})},
\]
with values in $[0,\infty]$.

\begin{lemma}[One-dimensional local domination]\label{lem:one-dimensional-domination}
Let $f,q \colon [0,1]\to\R$ have finite image.  Suppose that every
$x\in[0,1]$ has a relative neighbourhood $U_x$ such that
\[
  \abs{q(x)-q(\widetilde x)}
  \leq f(x)-f(\widetilde x)
  \qquad(\widetilde x\in U_x).
\]
Then
\[
  \Var(q)\leq\Var(f).
\]
\end{lemma}

\begin{proof}
The displayed inequality first shows that every point is a local maximum
of $f$.  If $\Var(f)=\infty$, there is nothing to prove.  If $f$ is
constant, the same inequality makes $q$ locally constant; connectedness
of $[0,1]$ then makes $q$ constant, and the conclusion is immediate.

Assume now that $f$ takes at least two values, and let
\[
  \kappa
  =
  \min\{
    \abs{u-v}
    \mid
    u,v\in f([0,1]),\ u\neq v
  \}
  >0.
\]
Let
\[
  N
  =
  \{x\in[0,1]\mid
    f\text{ is not locally constant at }x
  \}.
\]
We claim that $N$ is finite.  Indeed, let
\[
  x_1<\cdots<x_m
\]
be distinct points of $N$.  Choose pairwise disjoint relative open
intervals $I_j$ with $x_j\in I_j$, ordered so that
\[
  \sup I_j<\inf I_{j+1}
  \qquad(1\leq j<m).
\]
Since $f$ is not locally constant at $x_j$, there are
$a_j<b_j$ in $I_j$ such that
\[
  f(a_j)\neq f(b_j).
\]
Extending the ordered list
\[
  a_1<b_1<\cdots<a_m<b_m
\]
to a partition of $[0,1]$, we obtain
\[
  \Var(f)
  \geq
  \sum_{j=1}^m
  \abs{f(b_j)-f(a_j)}
  \geq m\kappa.
\]
Thus every finite subset of $N$ has cardinality at most
$\Var(f)/\kappa$, and therefore $N$ is finite.

On every connected component of $[0,1]\setminus N$, the function $f$ is
locally constant and hence constant.  On the same component, the local
domination inequality makes $q$ locally constant, and hence constant.

At a point $a\in N$, let
$f(a^-),f(a^+)$ and $q(a^-),q(a^+)$ denote the constant values on the
adjacent components of $[0,1]\setminus N$, whenever the corresponding
side exists.  Since $a$ is a local maximum of $f$, the local inequality
at $a$ gives
\[
  \abs{q(a)-q(a^-)}
  \leq f(a)-f(a^-),
  \qquad
  \abs{q(a)-q(a^+)}
  \leq f(a)-f(a^+).
\]
Write $N=\{c_1<\cdots<c_k\}$.  Form a partition $P$ by taking $0$, $1$,
all the points $c_i$, and one point in each nonempty connected component
of $[0,1]\setminus N$, with repetitions omitted.  Since both $f$ and $q$
are constant on each such component, this partition contains all their
possible changes.  If $V_P$ denotes the corresponding partition-variation
sum, then
\[
  V_P(f)=\Var(f),
  \qquad
  V_P(q)=\Var(q).
\]
The summands in these two partition variations are precisely the
one-sided jumps displayed above.  Comparing them term by term therefore
gives
\[
  \Var(q)\leq\Var(f). \qedhere
\]
\end{proof}

We next state the geometric estimate that connects slice variation to
neighbourhood growth.  Its self-contained proof is given in
\cref{app:minkowski-proof}.

For a compact set $E\subseteq[0,1]^2$, write
\[
  E^y=\{x\in[0,1] \mid (x,y)\in E\},
  \qquad
  E_x=\{y\in[0,1] \mid (x,y)\in E\},
\]
and
\[
  E_\eps^{[0,1]^2}=(E+[-\eps,\eps]^2)\cap[0,1]^2.
\]

\begin{proposition}[Relative anisotropic Minkowski-slicing estimate]
\label{prop:minkowski-slicing}
For every compact $E\subseteq[0,1]^2$,
\begin{equation}\label{eq:minkowski-slicing-main}
  \liminf_{\eps\downarrow0}
  \frac{\abs{E_\eps^{[0,1]^2}}-\abs{E}}{\eps}
  \geq
  \int_0^1\Var(\one_{E^y})\dd y
  +
  \int_0^1\Var(\one_{E_x})\dd x.
\end{equation}
\end{proposition}
The proof of \cref{prop:minkowski-slicing} is in \cref{app:minkowski-proof}.

\begin{remark}[Geometric meaning]
Both sides of \eqref{eq:minkowski-slicing-main} measure an anisotropic
one-dimensional boundary size.  For
$E=\{a\}\times[0,1]$ with $a\in(0,1)$, the left-hand side is $2$;
the integral of the horizontal slice variations is $2$, whereas the
vertical slice term is $0$.  For the diagonal segment from
$(0,0)$ to $(1,1)$, both horizontal and vertical slices contribute $2$,
and the left-hand side is $4$.  These examples explain the two separate
slice terms.  We only need the lower bound and make no general equality
claim.
\end{remark}

We also use the elementary coarea identity
\begin{equation}\label{eq:coarea-stated}
  \int_\R\Var(\one_{\{g\geq t\}})\dd t=\Var(g)
\end{equation}
for finite-image functions $g \colon [0,1]\to\R$.  A direct proof is included in
\cref{app:coarea}.

\begin{proof}[Proof of \cref{thm:analytic-engine}]
We first verify that the quantities $\Delta_v$ and $\Delta_h$ are
well-defined.  The local domination inequality implies that
\[
  \rho(\widetilde{\mathbf{z}})\leq \rho(\mathbf{z})
  \qquad(\widetilde{\mathbf{z}}\in U),
\]
so every point is a local maximum of $\rho$.  Hence $\rho$ is upper
semicontinuous and, in particular, Borel measurable.

For $a\in \rho([0,1]^2)$, put $E_a=\rho^{-1}(a)$.  Each $E_a$ is Borel.  Moreover,
if $\mathbf{z},\widetilde{\mathbf{z}}\in E_a$ and $\widetilde{\mathbf{z}}$ is sufficiently close to
$\mathbf{z}$, then
\[
  d_\infty(\gamma(\mathbf{z}),\gamma(\widetilde{\mathbf{z}}))
  \leq \rho(\mathbf{z})-\rho(\widetilde{\mathbf{z}})=0.
\]
Thus $\gamma|_{E_a}$ is locally constant, hence continuous.  Since $\rho([0,1]^2)$
is finite, the sets $E_a$ form a finite Borel partition of $[0,1]^2$, and it
follows that $\gamma$ is Borel measurable.  Its image is finite, so $\gamma$ is
bounded.  Therefore the boundary functions occurring in the
definitions of $\Delta_v$ and $\Delta_h$ are bounded Borel functions,
and the two Lebesgue integrals are finite.

Hence the superlevel sets
\[
  E_t=\{\mathbf{z}\in[0,1]^2 \mid \rho(\mathbf{z})\geq t\}
\]
are compact.

Because $M_\eps \rho\geq \rho$, the layer-cake identity gives
\[
  M_\eps \rho(\mathbf{z})-\rho(\mathbf{z})
  =
  \int_\R
  \bigl(\one_{\{M_\eps \rho\geq t\}}(\mathbf{z})
       -\one_{\{\rho\geq t\}}(\mathbf{z})\bigr)\dd t.
\]
Moreover,
\[
  \{M_\eps \rho\geq t\}=(E_t)_\eps^{[0,1]^2}.
\]
This set is compact, so $M_\eps \rho$ is upper semicontinuous and hence
measurable.  Since $\rho$ has finite image, only finitely many distinct
superlevel sets $E_t$ occur as $t$ varies.  For each of them, the
section-variation functions are Borel measurable by
\cref{lem:section-measurability}.  Thus the functions of $(t,x)$ and
$(t,y)$ occurring below are Borel measurable, which supplies all joint
measurability needed below.
For every $\eps>0$, Tonelli's theorem and the preceding layer-cake
identity give
\begin{align*}
  I_\eps
  &\coloneqq
  \int_{[0,1]^2}
  \frac{M_\eps\rho(\mathbf z)-\rho(\mathbf z)}{\eps}
  \dd\mathbf z\\
  &=
  \int_{\R}
  \frac{
    \abs{(E_t)_\eps^{[0,1]^2}}-\abs{E_t}
  }{\eps}
  \dd t.
\end{align*}

For complete precision, the following use
of Fatou's lemma may be read along an arbitrary sequence of positive
$\eps$ tending to zero; the resulting inequality for every such
sequence is equivalent to the displayed one-sided liminf.  Fatou's lemma and
\cref{prop:minkowski-slicing} therefore yield
\begin{align}
  \liminf_{\eps\downarrow0}
  \int_{[0,1]^2}\frac{M_\eps \rho-\rho}{\eps}\dd\mathbf{z}
  &\geq
  \int_\R\left(
    \int_0^1\Var(\one_{(E_t)^y})\dd y
    +
    \int_0^1\Var(\one_{(E_t)_x})\dd x
  \right)\dd t \notag\\
  &=
  \int_0^1\Var(\rho(\cdot,y))\dd y
  +
  \int_0^1\Var(\rho(x,\cdot))\dd x,
  \label{eq:analytic-to-slices}
\end{align}
where the last equality follows from Tonelli and the coarea identity
\eqref{eq:coarea-stated} on each slice.

Fix $y\in[0,1]$.  Applying
\cref{lem:one-dimensional-domination} to
\[
  x\longmapsto \rho(x,y),
  \qquad
  x\longmapsto \gamma_1(x,y),
\]
gives
\[
  \Var(\rho(\cdot,y))
  \geq
  \Var(\gamma_1(\cdot,y))
  \geq
  \abs{\gamma_1(1,y)-\gamma_1(0,y)}.
\]
After integration in $y$,
\[
  \int_0^1\Var(\rho(\cdot,y))\dd y\geq\Delta_v.
\]
Likewise,
\[
  \int_0^1\Var(\rho(x,\cdot))\dd x\geq\Delta_h.
\]
Together with \eqref{eq:analytic-to-slices}, this proves the first
assertion.

For the final assertion, the boundary assumption gives, for every
$x,y\in[0,1]$,
\[
  \abs{\gamma_1(1,y)-\gamma_1(0,y)}\geq1-2\delta,
  \qquad
  \abs{\gamma_2(x,1)-\gamma_2(x,0)}\geq1-2\delta.
\]
Hence $\Delta_v+\Delta_h\geq2(1-2\delta)$.
\end{proof}

\subsection{Growth across a finite slab}
\label{sec:obstruction}

We now pass from the infinitesimal estimate to a whole finite slab.  This
gives the second conclusion announced in \cref{sec:main-idea}: if a deflation
moves the lateral boundary only slightly at every depth, then its average
output radius on the bottom slice must be large.

\begin{theorem}[Boundary preservation forces radial growth]
\label{thm:boundary-radial-growth}
Let $\sigma\colon\Sqbot\to\Sqbot$ be a deflation, let $m>0$, and let
$\delta\geq0$.  Assume that $\sigma(\mathbf{c},r)$ is finite for every
\[
  (\mathbf{c},r)\in[0,1]^2\times[0,m].
\]
For $r\in[0,m]$, set
\[
  \rho_r(\mathbf{c})
  =
  \Rad(\sigma(\mathbf{c},r)),
  \qquad
  \gamma_r(\mathbf{c})
  =
  \Cen(\sigma(\mathbf{c},r)),
\]
and
\[
  A(r)
  =
  \int_{[0,1]^2}\rho_r(\mathbf{c})\dd\mathbf{c}.
\]
Suppose that the output centres move the lateral boundary by at most
$\delta$, in the sense that
\begin{equation}\label{eq:boundary-centre-control}
  d_\infty(\gamma_r(\mathbf{c}),\mathbf{c})
  \leq\delta
  \qquad
  \bigl(\mathbf{c}\in\partial([0,1]^2),\ 0\leq r\leq m\bigr).
\end{equation}
Then
\begin{equation}\label{eq:slab-average-growth}
  A(m)-A(0)\geq2m(1-2\delta).
\end{equation}
In particular,
\begin{equation}\label{eq:bottom-average-radius}
  \int_{[0,1]^2}
  \Rad(\sigma(\mathbf{c},m))\dd\mathbf{c}
  =
  A(m)
  \geq
  2m(1-2\delta),
\end{equation}
and
\begin{equation}\label{eq:bottom-average-excess}
  \int_{[0,1]^2}
  \bigl(\Rad(\sigma(\mathbf{c},m))-m\bigr)\dd\mathbf{c}
  \geq
  m(1-4\delta).
\end{equation}
Consequently, there is $\mathbf{c}\in[0,1]^2$ such that
\begin{equation}\label{eq:bottom-point-excess}
  \Rad(\sigma(\mathbf{c},m))-m
  \geq
  m(1-4\delta).
\end{equation}
Equivalently, for every $\beta<m(1-4\delta)$, the radius excess is greater
than $\beta$ at some point of the slice at depth $m$.
\end{theorem}

The pointwise lower bound is positive when $\delta<1/4$, which is the
regime used in the final argument.

\begin{proof}
Fix $r\in[0,m]$.  By \cref{lem:local-order}, every
$\mathbf{c}_0\in[0,1]^2$ has a relative neighbourhood on which
\[
  d_\infty\bigl(
    \gamma_r(\mathbf{c}_0),
    \gamma_r(\mathbf{c})
  \bigr)
  \leq
  \rho_r(\mathbf{c}_0)-\rho_r(\mathbf{c}).
\]
In particular, every point is a local maximum of $\rho_r$, so $\rho_r$
is upper semicontinuous and hence measurable.  Applying
\cref{thm:analytic-engine} and using
\eqref{eq:boundary-centre-control}, we obtain
\begin{equation}\label{eq:boundary-spatial-growth}
  \liminf_{\eta\downarrow0}
  \int_{[0,1]^2}
  \frac{M_\eta\rho_r(\mathbf{c})-\rho_r(\mathbf{c})}{\eta}
  \dd\mathbf{c}
  \geq
  2(1-2\delta).
\end{equation}

Since $\sigma$ has finite image and all outputs on the slab are finite,
the functions $\rho_r$ are uniformly bounded there.  Thus $A$ is
finite-valued.

The function $A$ is nondecreasing.  Indeed, if $r\leq s$, then
\[
  (\mathbf{c},s)\domle(\mathbf{c},r),
\]
and therefore
\[
  \rho_s(\mathbf{c})\geq\rho_r(\mathbf{c}).
\]
In particular, $A$ is Borel measurable.  Moreover,
\cref{lem:radial-upper-envelope} gives
\[
  \rho_{r+\eta}\geq M_\eta\rho_r
\]
whenever $0<\eta<m-r$.  It follows from
\eqref{eq:boundary-spatial-growth} that
\begin{equation}\label{eq:average-lower-derivative}
  \liminf_{\eta\downarrow0}
  \frac{A(r+\eta)-A(r)}{\eta}
  \geq
  2(1-2\delta)
  \qquad(0\leq r<m).
\end{equation}

Set
\[
  \eta_n=\frac{m}{n+1}
\]
and define on $[0,m]$
\[
  g_n(r)
  =
  \begin{cases}
    \dfrac{A(r+\eta_n)-A(r)}{\eta_n},
      &0\leq r\leq m-\eta_n,\\[1ex]
    0,
      &m-\eta_n<r\leq m.
  \end{cases}
\]
The functions $g_n$ are measurable and nonnegative.  Monotonicity of
$A$ gives
\begin{align*}
  \int_0^m g_n(r)\dd r
  &=
  \frac1{\eta_n}
  \left(
    \int_{m-\eta_n}^m A(s)\dd s
    -
    \int_0^{\eta_n}A(s)\dd s
  \right)\\
  &\leq
  A(m)-A(0).
\end{align*}
By \eqref{eq:average-lower-derivative}, for every $r<m$,
\[
  \liminf_{n\to\infty}g_n(r)
  \geq
  2(1-2\delta).
\]
Fatou's lemma therefore yields
\[
  2m(1-2\delta)
  \leq
  \int_0^m\liminf_{n\to\infty}g_n(r)\dd r
  \leq
  \liminf_{n\to\infty}\int_0^m g_n(r)\dd r
  \leq
  A(m)-A(0),
\]
which proves \eqref{eq:slab-average-growth}.

All output radii are nonnegative, so $A(0)\geq0$.  This gives
\eqref{eq:bottom-average-radius}, and subtracting $m$ gives
\eqref{eq:bottom-average-excess}.  Finally, $\rho_m$ has finite image, so
it attains a maximum.  Since $[0,1]^2$ has area $1$, its maximum is at
least its average.  This proves \eqref{eq:bottom-point-excess}.
\end{proof}

It is worth recording the resulting unconditional quantitative
obstruction.  It no longer assumes boundary control explicitly.

\begin{corollary}[Quantitative radius obstruction]
\label{thm:radius-obstruction}
Let $\sigma\colon\Sqbot\to\Sqbot$ be a deflation, and let $m\geq0$.
Then there are $\mathbf{c}\in[0,1]^2$ and $r\in[0,m]$ such that
\[
  \Rad(\sigma(\mathbf{c},r))-r
  \geq
  \frac{m}{4m+1}.
\]
\end{corollary}

\begin{proof}
The case $m=0$ follows from the nonnegativity of the radius excess.
Assume $m>0$, and put
\[
  \alpha_m=\frac{m}{4m+1}.
\]
If the conclusion failed, then
\[
  \Rad(\sigma(\mathbf{c},r))-r<\alpha_m
  \qquad
  (\mathbf{c}\in[0,1]^2,\ 0\leq r\leq m).
\]
All outputs on the slab would therefore be finite, and
\eqref{eq:basic-center-excess} would give
\[
  d_\infty\bigl(
    \Cen(\sigma(\mathbf{c},r)),
    \mathbf{c}
  \bigr)
  \leq\alpha_m
\]
there.  By \cref{thm:boundary-radial-growth}, some point of the slice at
depth $m$ would then have radius excess at least
\[
  m(1-4\alpha_m)=\alpha_m,
\]
a contradiction.
\end{proof}

\section{No approximate identity of deflations}
\label{sec:main-conclusion}

We can now put the two parts of the argument together.
\Cref{thm:uniform-slab-approximation} says that an approximate identity
contains a member with uniformly small radius excess on any prescribed
finite slab.  By contrast, \cref{thm:radius-obstruction} gives a fixed
positive lower bound for the radius excess of every deflation somewhere
on that slab.

\begin{theorem}[Incompatibility of the two estimates]
\label{thm:no-deflation-approximate-identity}
The formal-square domain $\Sqbot$ carries no approximate identity
consisting of deflations.
\end{theorem}

\begin{proof}
Suppose that $\mathcal A$ is an approximate identity on $\Sqbot$
consisting of deflations.  Fix $m>0$ and choose
\[
  0<\eps<\frac{m}{4m+1}.
\]
By \cref{thm:uniform-slab-approximation}, there is
$\sigma\in\mathcal A$ such that
\[
  0
  \leq
  \Rad(\sigma(\mathbf{c},r))-r
  \leq
  \eps
  \qquad
  (\mathbf{c}\in[0,1]^2,\ 0\leq r\leq m).
\]
Since $\sigma$ is a deflation, \cref{thm:radius-obstruction} gives
$\mathbf{c}\in[0,1]^2$ and $r\in[0,m]$ such that
\[
  \Rad(\sigma(\mathbf{c},r))-r
  \geq
  \frac{m}{4m+1}
  >
  \eps,
\]
contradicting the uniform radius-excess bound on the slab.
\end{proof}

\begin{theorem}\label{thm:not-RB}
The formal-square domain $\Sqbot$ is not an RB-domain.
\end{theorem}

\begin{proof}
If $\Sqbot$ were an RB-domain, then
\cref{prop:RB-char} would give an approximate identity consisting of
deflations, contrary to
\cref{thm:no-deflation-approximate-identity}.
\end{proof}

\begin{proof}[Proof of \cref{thm:main-intro}]
Combine \cref{prop:FS,thm:not-RB}.
\end{proof}

\appendix

\section{One-dimensional parallel sets and the slicing estimate}
\label[appendix]{app:analysis}

This appendix supplies the technical facts used in
\cref{sec:analytic}.  Their order mirrors the logic of the proof: first a
one-dimensional parallel-set formula, then lower semicontinuity and
coarea, and finally the two-dimensional slicing estimate.

\subsection{Parallel sets in one dimension}

For a compact set $A\subseteq[0,1]$ and $s\geq0$, set
\[
  A_s=(A+[-s,s])\cap[0,1].
\]
Lebesgue measure is denoted by $\abs{\cdot}$.

\begin{lemma}[Exact one-dimensional parallel-set formula]
\label{lem:one-dimensional-parallel}
For every compact $A\subseteq[0,1]$ and every $\rho\geq0$,
\begin{equation}\label{eq:one-dimensional-exact}
  \abs{A_\rho}-\abs{A}
  =\int_0^\rho\Var(\one_{A_s})\dd s.
\end{equation}
Moreover,
\begin{equation}\label{eq:one-dimensional-limit}
  \Var(\one_A)
  =\lim_{\rho\downarrow0}
    \frac{\abs{A_\rho}-\abs{A}}{\rho},
\end{equation}
where the value $+\infty$ is allowed.
\end{lemma}

\begin{proof}
The case $A=\varnothing$ is immediate.  Assume $A\neq\varnothing$.  The
relative complement $[0,1]\setminus A$ is a countable disjoint union of
relative open intervals.  Call a component $(a,b)$ an internal gap; the
possible components $[0,b)$ and $(a,1]$ are boundary gaps.  If a gap has
length $\ell$, enlargement by $s$ removes length $\min(2s,\ell)$ from an
internal gap and length $\min(s,\ell)$ from a boundary gap.  Hence
\begin{equation}\label{eq:gap-growth}
  \abs{A_s}-\abs{A}
  =
  \sum_{G\in\mathcal G_i}\min(2s,\abs{G})
  +
  \sum_{G\in\mathcal G_b}\min(s,\abs{G}),
\end{equation}
where $\mathcal G_i$ and $\mathcal G_b$ are the internal and boundary
gaps.

For fixed $s>0$, an internal gap with $\abs{G}>2s$ contributes two jumps
to $\one_{A_s}$, and a boundary gap with $\abs{G}>s$ contributes one.
Thus
\[
  \Var(\one_{A_s})
  =
  2\#\{G\in\mathcal G_i \mid \abs{G}>2s\}
  +
  \#\{G\in\mathcal G_b \mid \abs{G}>s\}.
\]
Tonelli's theorem for nonnegative series gives
\begin{align*}
  \int_0^\rho\Var(\one_{A_s})\dd s
  &=
  \sum_{G\in\mathcal G_i}
     \int_0^\rho2\one_{\{\abs{G}>2s\}}\dd s
  +
  \sum_{G\in\mathcal G_b}
     \int_0^\rho\one_{\{\abs{G}>s\}}\dd s\\
  &=
  \sum_{G\in\mathcal G_i}\min(2\rho,\abs{G})
  +
  \sum_{G\in\mathcal G_b}\min(\rho,\abs{G}),
\end{align*}
which is \eqref{eq:one-dimensional-exact} by
\eqref{eq:gap-growth}.

Dividing \eqref{eq:gap-growth} by $s$ gives
\[
  \frac{\abs{A_s}-\abs{A}}s
  =
  \sum_{G\in\mathcal G_i}\min(2,\abs{G}/s)
  +
  \sum_{G\in\mathcal G_b}\min(1,\abs{G}/s).
\]
As $s\downarrow0$, the summands increase to $2$ for each internal gap and
$1$ for each boundary gap.  Their sum is exactly
$\Var(\one_A)$, possibly infinite.  Monotone convergence proves
\eqref{eq:one-dimensional-limit}.
\end{proof}

\begin{lemma}[Lower semicontinuity under decreasing compact approximation]
\label{lem:lsc-variation}
Let $A_1\supseteq A_2\supseteq\cdots$ be compact subsets of $[0,1]$ and
let $A=\bigcap_nA_n$.  Then
\[
  \Var(\one_A)
  \leq\liminf_{n\to\infty}\Var(\one_{A_n}).
\]
The same conclusion holds for a decreasing family $A_\eps\downarrow A$
as $\eps\downarrow0$.
\end{lemma}

\begin{proof}
The indicators $\one_{A_n}$ converge pointwise to $\one_A$.  For every
finite partition $0=t_0<\cdots<t_k=1$,
\begin{align*}
  \sum_{j=1}^k
  \abs{\one_A(t_j)-\one_A(t_{j-1})}
  &=
  \lim_{n\to\infty}
  \sum_{j=1}^k
  \abs{\one_{A_n}(t_j)-\one_{A_n}(t_{j-1})}\\
  &\leq
  \liminf_{n\to\infty}\Var(\one_{A_n}).
\end{align*}
Take the supremum over partitions.  The one-parameter statement follows
by applying the sequential result to arbitrary sequences
$\eps_n\downarrow0$.
\end{proof}

\subsection{Coarea for finite-image functions}\label{app:coarea}

\begin{lemma}[One-dimensional coarea]
\label{lem:coarea}
If $g \colon [0,1]\to\R$ has finite image, then
\[
  \int_\R\Var(\one_{\{g\geq t\}})\dd t=\Var(g),
\]
with the convention that both sides may be infinite.
\end{lemma}

\begin{proof}
For a finite partition $P:0=x_0<\cdots<x_n=1$, write
\[
  V_P(g)=\sum_{j=1}^n\abs{g(x_j)-g(x_{j-1})}.
\]
The elementary identity
\[
  \abs{a-b}
  =\int_\R\abs{\one_{\{a\geq t\}}-\one_{\{b\geq t\}}}\dd t
\]
gives
\[
  V_P(g)
  =\int_\R V_P(\one_{\{g\geq t\}})\dd t
  \leq\int_\R\Var(\one_{\{g\geq t\}})\dd t.
\]
Taking the supremum over $P$ proves one inequality.

Let the distinct values of $g$ be $a_1<\cdots<a_m$.  For
$1\leq i<m$, put
\[
  h_i=\one_{\{g\geq a_{i+1}\}},
  \qquad
  \Delta_i=a_{i+1}-a_i.
\]
For every partition $P$, the layer decomposition of the finitely many
values gives
\[
  V_P(g)=\sum_{i=1}^{m-1}\Delta_i V_P(h_i),
\]
and, up to endpoints of the level intervals,
\[
  \int_\R\Var(\one_{\{g\geq t\}})\dd t
  =\sum_{i=1}^{m-1}\Delta_i\Var(h_i).
\]
If all these variations are finite, choose for each $i$ a partition
whose $h_i$-variation is arbitrarily close to $\Var(h_i)$ and take a
common refinement.  The preceding identity for $V_P(g)$ then shows
\[
  \Var(g)\geq\sum_{i=1}^{m-1}\Delta_i\Var(h_i).
\]
If some $\Var(h_i)$ is infinite, partitions with arbitrarily large
$h_i$-variation instead show that $\Var(g)=\infty$.  This proves the
reverse inequality in all cases.
\end{proof}

\subsection{Measurability of section variations}\label{app:measurability}

\begin{lemma}[Joint measurability of section variations]
\label{lem:section-measurability}
Let $F\subseteq[0,1]^2$ be compact and, for $a,b\geq0$, put
\[
  F_{a,b}=\bigl(F+[-a,a]\times[-b,b]\bigr)\cap[0,1]^2.
\]
Then the maps
\[
  (a,b,y)\longmapsto\Var(\one_{(F_{a,b})^y}),
  \qquad
  (a,b,x)\longmapsto\Var(\one_{(F_{a,b})_x})
\]
from $[0,\infty)^2\times[0,1]$ to $[0,\infty]$ are Borel measurable.
In particular, the section-variation functions of a fixed compact set
are measurable.
\end{lemma}

\begin{proof}
For $s\geq0$, define
\[
  H(a,b,y,s)=\abs{((F_{a,b})^y)_s}.
\]
The incidence set of all $(\xi,a,b,y,s)$ for which there exist
$(u,v)\in F$ and $w\in[0,1]$ satisfying
\[
  \abs{w-u}\leq a,
  \qquad
  \abs{y-v}\leq b,
  \qquad
  \abs{\xi-w}\leq s
\]
is closed: witnesses can be passed to a convergent subsequence in the
compact set $F\times[0,1]$.  Its $(a,b,y,s)$-section in the variable
$\xi$ is exactly $((F_{a,b})^y)_s$.  Integrating its Borel indicator in
$\xi$ shows that $H$ is Borel measurable.  By
\eqref{eq:one-dimensional-limit},
\[
  \Var(\one_{(F_{a,b})^y})
  =\lim_{n\to\infty}
  n\bigl(H(a,b,y,1/n)-H(a,b,y,0)\bigr),
\]
so the first map is measurable.  Interchanging the two coordinates gives
the second map.
\end{proof}

\subsection{The two-dimensional slicing estimate}
\label{app:minkowski-proof}

For $a,b\geq0$, set
\[
  E_{a,b}
  =\bigl(E+[-a,a]\times[-b,b]\bigr)\cap[0,1]^2.
\]
Thus $E_\eps^{[0,1]^2}=E_{\eps,\eps}$.

\begin{proof}[Proof of \cref{prop:minkowski-slicing}]
Decompose the square enlargement into a horizontal enlargement followed
by a vertical enlargement:
\begin{equation}\label{eq:enlargement-decomposition}
  \abs{E_{\eps,\eps}}-\abs{E}
  =
  \bigl(\abs{E_{\eps,0}}-\abs{E}\bigr)
  +
  \bigl(\abs{E_{\eps,\eps}}-\abs{E_{\eps,0}}\bigr).
\end{equation}
We estimate the two terms separately.

For every $y\in[0,1]$,
\[
  (E_{\eps,0})^y=(E^y)_\eps.
\]
Fubini and \cref{lem:one-dimensional-parallel} give
\begin{align*}
  \abs{E_{\eps,0}}-\abs{E}
  &=
  \int_0^1
  \bigl(\abs{(E^y)_\eps}-\abs{E^y}\bigr)\dd y\\
  &=
  \int_0^1\int_0^\eps
  \Var(\one_{(E^y)_s})\dd s\dd y.
\end{align*}
After the change of variables $s=\tau\eps$,
\[
  \frac{\abs{E_{\eps,0}}-\abs{E}}\eps
  =
  \int_0^1\int_0^1
  \Var(\one_{(E^y)_{\tau\eps}})\dd\tau\dd y.
\]
For fixed $y$ and $\tau$, the compact sets
$(E^y)_{\tau\eps}$ decrease to $E^y$ as $\eps\downarrow0$.
By \cref{lem:lsc-variation} and Fatou's lemma,
\begin{equation}\label{eq:horizontal-slicing}
  \liminf_{\eps\downarrow0}
  \frac{\abs{E_{\eps,0}}-\abs{E}}\eps
  \geq
  \int_0^1\Var(\one_{E^y})\dd y.
\end{equation}
The measurability needed here is supplied by
\cref{lem:section-measurability}.

For the second term, fix $x\in[0,1]$.  Enlarging the vertical section of
$E_{\eps,0}$ by $s$ gives
\[
  ((E_{\eps,0})_x)_s=(E_{\eps,s})_x.
\]
Fubini and \cref{lem:one-dimensional-parallel} now give
\begin{align*}
  \abs{E_{\eps,\eps}}-\abs{E_{\eps,0}}
  &=
  \int_0^1
  \bigl(\abs{(E_{\eps,\eps})_x}
       -\abs{(E_{\eps,0})_x}\bigr)\dd x\\
  &=
  \int_0^1\int_0^\eps
  \Var(\one_{(E_{\eps,s})_x})\dd s\dd x.
\end{align*}
With $s=\tau\eps$,
\[
  \frac{\abs{E_{\eps,\eps}}-\abs{E_{\eps,0}}}\eps
  =
  \int_0^1\int_0^1
  \Var(\one_{(E_{\eps,\tau\eps})_x})\dd\tau\dd x.
\]
For fixed $x$ and $\tau$, the compact sets
$(E_{\eps,\tau\eps})_x$ decrease to $E_x$ as $\eps\downarrow0$.  The
inclusion $E_x\subseteq(E_{\eps,\tau\eps})_x$ is immediate.  Conversely,
if $y$ belongs to all these sections, choose $\eps_n\downarrow0$ and
$(x_n,y_n)\in E$ such that
\[
  \abs{x_n-x}\leq\eps_n,
  \qquad
  \abs{y_n-y}\leq\tau\eps_n.
\]
Then $(x_n,y_n)\to(x,y)$, and compactness of $E$ gives $(x,y)\in E$.
Thus $y\in E_x$.  Another application of
\cref{lem:lsc-variation}, Fatou's lemma, and
\cref{lem:section-measurability} yields
\begin{equation}\label{eq:vertical-slicing}
  \liminf_{\eps\downarrow0}
  \frac{\abs{E_{\eps,\eps}}-\abs{E_{\eps,0}}}\eps
  \geq
  \int_0^1\Var(\one_{E_x})\dd x.
\end{equation}
Combining \eqref{eq:enlargement-decomposition},
\eqref{eq:horizontal-slicing}, and \eqref{eq:vertical-slicing} proves
\eqref{eq:minkowski-slicing-main}.
\end{proof}

\section*{Acknowledgments}
The author used \mbox{ChatGPT} (OpenAI) as a research and writing assistant during the development of this manuscript. 
I had been developing the underlying proof strategy for some time and, on
30 April 2026, asked ChatGPT to help work out the most technically demanding
analytic parts.  Through a back-and-forth discussion, the remaining gaps
were filled and a complete proof emerged on the same date.

The choice of the formal-square model for this problem, the central
bounded-variation strategy, and the overall architecture of the argument
originated with the author.
\mbox{ChatGPT} provided substantial assistance in proving the relative anisotropic Minkowski-slicing estimate (\cref{prop:minkowski-slicing}), in working out and stress-testing technical details, in identifying and repairing gaps in intermediate arguments, and improving the organization and exposition of the manuscript. The author assumes full responsibility for the contents.

\begingroup
\renewcommand{\bibliofont}{\small}
\bibliographystyle{alpha}
\bibliography{FS_vs_RB}

@Article{Lawson2008,
  author     = {Lawson, Jimmie D.},
  journal    = {Theoret. Comput. Sci.},
  title      = {Metric spaces and {$FS$}-domains},
  year       = {2008},
  number     = {1--2},
  pages      = {73--74},
  volume     = {405},
  doi        = {10.1016/j.tcs.2008.06.026}
}

@InCollection{Jung1990,
  author    = {Jung, Achim},
  booktitle = {Fifth Annual IEEE Symposium on Logic in Computer Science},
  publisher = {IEEE Computer Society Press},
  title     = {The classification of continuous domains (extended abstract)},
  year      = {1990},
  pages     = {35--40},
  doi       = {10.1109/LICS.1990.113731}
}

@Misc{ChenCone2026,
  author        = {Chen, Yuxu},
  title         = {Cone domains separate {FS}-domains from {RB}-domains},
  year          = {2026},
  archiveprefix = {arXiv},
  eprint        = {2607.02251},
  primaryclass  = {math.GN},
  note          = {arXiv:2607.02251}
}

@InCollection{AbramskyJung1994,
  author    = {Abramsky, Samson and Jung, Achim},
  booktitle = {Handbook of Logic in Computer Science},
  publisher = {Oxford University Press},
  title     = {Domain theory},
  year      = {1994},
  pages     = {1--168},
  volume    = {3}
}

@Book{GierzHofmannEtAl2003,
  author    = {Gierz, G. and Hofmann, K. H. and Keimel, K. and Lawson, J. D. and Mislove, M. and Scott, D. S.},
  publisher = {Cambridge University Press},
  title     = {Continuous Lattices and Domains},
  year      = {2003},
  series    = {Encyclopedia of Mathematics and its Applications},
  volume    = {93},
  doi       = {10.1017/CBO9780511542725}
}

@Article{ZouLiGuo2018,
  author  = {Zou, Z. and Li, Q. and Guo, L.},
  journal = {Categories and General Algebraic Structures with Applications},
  title   = {A note on the problem when {FS}-domains coincide with {RB}-domains},
  year    = {2018},
  number  = {1},
  pages   = {51--59},
  volume  = {8},
  doi     = {10.29252/cgasa.8.1.51}
}

@Misc{ChenKouLyu2026,
  author        = {Chen, Yuxu and Kou, Hui and Lyu, Zhenchao},
  title         = {{FS-domains are not always RB-domains}},
  year          = {2026},
  archiveprefix = {arXiv},
  eprint        = {2607.00568},
  primaryclass  = {math.GN},
  note          = {arXiv:2607.00568}
}
\endgroup

\end{document}